\documentclass[a4paper,11pt]{amsart}

\usepackage{amstext,amsfonts,amsthm,graphicx,amssymb,amscd,epsfig}
\usepackage{amsmath}
\usepackage[ansinew]{inputenc}    

\usepackage{mathrsfs}

\theoremstyle{definition}
\newtheorem{definition}{Definition}[section]
\newtheorem{ex}[definition]{Example}

\theoremstyle{plain}
\newtheorem{prop}[definition]{Proposition}
\newtheorem{lem}[definition]{Lemma}

\newtheorem{teo}[definition]{Theorem}

\newfont{\bbb}{msbm10 scaled\magstephalf}     
\def\C{\mathbb C}
\def\K{\mathbb K}
\def\R{\mathbb R}

\def\R{\mbox{\bbb R}}

\def\O{\mathcal O}
\def\A{\mathscr A}

\def\Aecod{\operatorname{\mbox{$\A_e$-cod}}}
\def\Keq{\mathscr{K}}
\def\Kecod{\operatorname{\mbox{$\Keq_e$-cod}}}

\def\Kecod{\operatorname{\mbox{$\mathscr{K}_e$-cod}}}

\usepackage[usenames]{color}

\newcommand{\tae}{T\A_e}

\newcommand{\nae}{N\A_e}

\newcommand{\tke}{T\mathscr{K}_e}

\newcommand{\nke}{N\mathscr{K}_e}
\newcommand{\fpdat}[3]{\left. \frac{\partial #1}{\partial #2} \right|_{#3}}
\newcommand{\dpar}[2]{\frac{\partial #1}{\partial #2}}
\def\m{\mathfrak{m}}

\title{A note on 1-parameter stable unfoldings}

\author{I. Breva Ribes, R. Oset Sinha}

\date{}

\address{Departament de Matem\`atiques,
Universitat de Val\`encia, Campus de Burjassot, 46100 Burjassot,
Spain}

\email{raul.oset@uv.es}

\email{igbreri@alumni.uv.es}

\thanks{Work of R. Oset Sinha partially supported by Grant PID2021-124577NB-I00 funded by MCIN/AEI/ 10.13039/501100011033 and by ``ERDF A way of making Europe"}


\begin{document}
\begin{abstract}
We give two characterisations of when a map-germ admits a 1-parameter stable unfolding, one related to the $\mathscr K_e$-codimension and another related to the normal form of a versal unfolding. We then prove that there are infinitely many finitely determined map-germs of multiplicity 4 from $\K^3$ to $\K^3$ which do not admit a 1-parameter stable unfolding.
\end{abstract}

\maketitle

\emph{Dedicated to Maria Aparecida Soares Ruas (Cidinha) on the ocassion of her 75th birthday.}

\section{Introduction}

In classification problems of map-germs and in the study of their algebraic or topological invariants, having a 1-parameter stable unfolding is a desirable property.

\begin{definition}\label{def_opsu}
Let $f\colon(\K^n,0)\to(\K^p,0)$ be a smooth map-germ.
A 1-parameter stable unfolding (OPSU) of $f$ is a smooth map-germ $F\colon(\K^{n+1},0)\to(\K^{p+1},0)$ of the form $F(x,\lambda) = (f_\lambda(x),\lambda)$, with $f_0 = f$, which is stable as a map-germ.
\end{definition}

Notice that a germ may have arbitrarily high $\mathscr A_e$-codimension but still admit an OPSU. For example, the germs $f_k(x,y)=(x,y^3+x^{k+1}y)$ in Rieger's list (\cite{Rieger}) have $\mathscr A_e$-codimension $k$ but admit the OPSU $F_k(x,y,\lambda)=(x,y^3+x^{k+1}y+\lambda y,\lambda)$, which is a cuspidal edge in $\K^3$. Many papers throughout the literature need the hypothesis of a germ admitting an OPSU in their main theorems. For example, in \cite{NORW} it is the class of germs in which a method to calculate liftable vector fields is applicable, in Theorem 2.19 in \cite{GCNB} it is related to the Mond conjecture and in \cite{Houston} it is necessary for a characterisation of which map-germs are augmentations. However, it is not fully understood when a map-germ admits an OPSU or not. There are obvious constraints given by the maximum possible multiplicity of stable germs in one dimension above, i.e. multiplicity 6 germs from $\K^3$ to $\K^3$ cannot have OPSU since there is no stable germ of multiplicity 6 in $\C^4$ to $\C^4$. Amongst explicitly given classifications of corank 1 simple germs most germs seem to admit OPSU, for example, in Marar and Tari's classification of simple corank 1 germs from $\K^3$ to $\K^3$ (\cite{marartari}) all admit OPSU. However, this is not true in general, for example $(x,y^4+x^2y)$ in Rieger's list is simple but does not admit an OPSU. Also, the germs $H_k(x,y)=(x,y^3,y^{3k-1}+xy)$ in Mond's list (\cite{mond}) are simple of corank 1 and do not admit OPSU.

In this note we prove two characterisations of when germs admit an OPSU. The first one related to the $\mathscr K_e$-codimension is well known but we have not found a proof of it in the literature, so we include it for the sake of completeness. The second one is related to the form of a versal unfolding and we believe it can be useful in several different contexts. Our results are in fact more general and related to the existence of stable unfoldings, the statements for OPSUs come as corollaries. We then turn our attention to the case of corank 1 map-germs from $\K^3$ to $\K^3$. Any multiplicity 2 germ in these dimensions is stable, any multiplicity 3 germ will be of the form $(x,y,z^3+h(x,y,z))$ with $h\in \m_3$ and it admits an OPSU $(x,y,z^3+h(x,y,z)+\lambda z,\lambda)$. By \cite{marartari} all simple multiplicity 4 germs admit an OPSU too. We prove that there are infinitely many non-equivalent finitely determined map-germs of multiplicity 4 which do not admit an OPSU.

{{\it Acknowledgements:} The authors thank M. A. S. Ruas for constant encouragement and inspiration. In fact, all the ideas in this paper have been directly inspired by discussions with Ruas.}

\section{Preliminaries}

Let $\K$ be either $\C$ or $\R$.
Denote by $\O_d$ the local ring of germs of smooth functions in $d$ variables over $\K$, and denote its maximal ideal by $\m_d$.
We write $\theta_d = \O_d\times\overset{d}\cdots \times \O_d$.
If  $f\colon(\K^n,0)\to(\K^p,0)$ is a smooth (holomorphic or $C^\infty$ in $\C$ and $\R$, respectively) map-germ, then we define $\theta(f) = \O_n\times\overset{p}\cdots\times \O_n$.

Recall that if $\operatorname{Diff}(\K^d,0)$ is the group of germs of diffeomorphisms in $\K^d$, then the group $\A = \operatorname{Diff}(\K^n,0)\times\operatorname{Diff}(\K^p,0)$ acts on the set of map-germs $f\colon(\K^n,0)\to(\K^p,0)$ via the natural compositions.
The equivalence relation defined by the orbits of this action is called $\A$-equivalence.
We can assign to each $f$ the following spaces:
\begin{equation*}
\begin{aligned}
\tae f  & = tf(\theta_n) + \omega f (\theta_p), \\
\nae f  & = \frac{\theta(f)}{tf(\theta_n) + \omega f (\theta_p)}
\end{aligned}
\end{equation*}
where $tf$ acts as the composition with the differential $df$, and $\omega f$ acts as the pre-composition with $f$.
These spaces are respectively called the $\A_e$-tangent space to $f$ and the $\A_e$-normal space to $f$.

If $\Keq$ is the subgroup of $\operatorname{Diff}(\K^{n+p},0)$ given by diffeomorphisms of the form:
$$\Phi(x,y) = (\phi(x),\psi(x,y))$$
with $\psi(x,0) = 0$ for all $x$,
then $\Keq$ acts on map-germs $f\colon(\K^n,0)\to(\K^p,0)$ such that for each $\Phi\in\Keq$:
$$\Phi\cdot f (\phi(x)) = \psi(x,f(x))$$
The induced equivalence relation is called $\Keq$-equivalence or \textit{contact equivalence}, and we can also assign to each $f$ the following spaces:
\begin{equation*}
\begin{aligned}
\tke f  & = tf(\theta_n) + f^*\m_p\theta(f), \\
\nke f  & = \frac{\theta(f)}{tf(\theta_n) + f^*\m_p\theta(f)}
\end{aligned}
\end{equation*}
where $f^*\m_p$ is the ideal generated by the components of $f$ over $\O_n$.
These are called the $\Keq_e$-tangent space to $f$ and the $\Keq_e$-normal space to $f$.

We define the $\A_e$-codimension of $f$ as the dimension over $\K$ of $\nae f$, and we denote it by $\Aecod(f)$. A map-germ is said to be $\A$-finite if it has finite $\A_e$-codimension.
Being stable is equivalent to having $\A_e$-codimension equal to 0.
The $\Keq_e$-codimension is defined similarly, and if it is finite we say that $f$ has finite singularity type.

Being $\A$-finite is equivalent to being finitely $\A$-determined.
We recall that, in the analytic case, this means that for any $\A$-finite map-germ $f$ there exists a $d\in\mathbb N$ such that, if the Taylor expansion of another map-germ $g$ coincides with the expansion of $f$ up to degree $d$, then both map-germs are $\A$-equivalent.

The multiplicity of $f$ is the dimension over $\K$ of $\frac{\O_n}{f^*\m_p}$ and it is constant along the $\Keq$-orbit.
Finally, recall:
\begin{definition}
A smooth $f\colon(\K^n,0)\to(\K^p,0)$ is $\A$-simple if there exist a finite number of $\A$-equivalence classes such that, if the versal unfolding of $f$ admits a representative $F\colon U\to V $ with $U\subseteq \K^n\times\K^d, V\subseteq\K^p\times\K^d$, of the form $F(x,\lambda) = (f_\lambda(x),\lambda)$, for each $(y,\lambda)\in V$ the map-germ $f_\lambda\colon(\K^n,f_\lambda^{-1}(y))\to(\K^p,y)$ lies in one of those classes.
\end{definition}

We refer to \cite{nunomond} for more details on all these definitions.

\section{Minimal stable unfoldings}

In \cite{matherIV}, Mather gave a method to obtain stable mappings as unfoldings of rank 0 map-germs and proved that any stable germ can be obtained by that method. This method provided a recipe to construct a stable unfolding for any given map-germ of finite $\Keq_e$-codimension.
The procedure can be summarized as follows: if $f\colon(\K^n,0)\to(\K^p,0)$ is of rank $r$ and has finite $\Keq_e$-codimension, then it can be seen as the unfolding of some rank 0 map-germ, $f_0\colon(\K^{n-r},0)\to(\K^{p-r},0)$ of the same codimension $\Kecod(f_0) = m < \infty$.
Since $f_0$ is of rank 0, $\tke f_0\subseteq \m_{n-r}\theta(f_0)$.
Therefore, one can find $\gamma_1,\ldots,\gamma_{m-p+r}\in\theta(f_0)$ such that:
$$\frac{\m_{n-r}\theta(f_0)}{\tke f_0} = \operatorname{Sp}_\K\lbrace\gamma_1,\ldots,\gamma_{m-p+r}\rbrace$$
Mather's method now ensures that the unfolding:
$$\left( f(x) +\sum_{i=1}^{m-p+r} u_i\gamma_i,u_1,\ldots,u_{m-p+r}\right)$$
is stable.
Here the $\gamma_i$ are seen as vector fields in $\theta(f)$ via the natural inclusion.
For detailed explanations and proofs on the results that support this method, we suggest reading Section 7.2 of \cite{nunomond}.

This procedure has a minor inconvenience, namely that the number of parameters required to obtain the stable unfolding might be excessive.
The following example shows a na\"ive approach to this situation, but hints at what might happen when one is working with a more intricate map-germ:

\begin{ex}
Let $f(x,y) = (x,y^4 + xy)$.
We want to apply Mather's method in order to obtain a stable unfolding.
Here, $f$ can be seen as an unfolding of $f_0(y) = y^4$.
Now:
$$\frac{\m_1\theta(f_0)}{\tke f_0} = \operatorname{Sp}_\K\lbrace y, y^2\rbrace$$
Hence a stable unfolding of $f$ is:
$$F(x,y,u_1,u_2) = \left( x,y^4 + xy + u_1y+u_2y^2,u_1,u_2\right)$$
Clearly the term $u_1y$ is redundant, in the sense that we only needed to add the term $u_2y^2$ in order to obtain a stable unfolding (see the stable map-germs from $\K^3\to\K^3$ in page 210 in \cite{gibson}, for example).
\end{ex}

It is widely known that to avoid this problem of repetition one must take out of the computation those elements that, once seen as vector fields in $\theta(f)$, lie on the same  class in $\nke f$ as some constant vector field.
For instance, in our example, the vector field $(0,y)$ lies in the same class as $(-1,0)$ in $\nke f$, since $(1,y)$ is in $\tke f$.

We have not been able to find in the literature a published proof of this detail.
The following discussion aims to provide a formal proof for future reference.

In order to do this consider the following quotient, which is just a version of the $\Keq_e$-normal space with the constant vector fields added into the zero class:
$$N(f) = \frac{\theta(f)}{tf(\theta_n)+f^*\m_p\theta(f) + \omega f(\theta_p)}$$
This space was studied in depth by Ruas in her thesis (\cite{ruasthesis}), using it to provide criteria for determining when an $\A$-orbit is open on its $\Keq$-orbit.
As we will see, it is also the right tool to use when trying to build a stable unfolding with the minimal number of parameters.

\begin{definition}
We will say that a $d$-parameter stable unfolding $F$ of $f$ is a {\it minimal stable unfolding} if no other stable unfolding exists with less than $d$ parameters.
\end{definition}

Denote by $\dpar{}{X_1},\ldots,\dpar{}{X_p}$ the constant vector fields in $\theta(f)$.
Then we can also denote the constant vector fields of $\theta(F)$ by $\dpar{}{X_1},\ldots,\dpar{}{X_p},\dpar{}{U_1},\ldots,\dpar{}{U_d}$.
The following result can be found as Lemma 5.5 in \cite{nunomond}:
\begin{lem}\label{lema_k_isom}
For any unfolding $F(x,u) = (f_u(x),u)$ of $f$, there is an isomorphism $\beta\colon\nke F\to\nke f$ that takes the class of $\dpar{}{X_i}$ to the class in $\nke f$ of $\dpar{}{X_i}$ for each $i=1,\ldots,p$, and the class of $\dpar{}{U_j}$ to the class in $\nke f$ of $ \dot{F_j} = \fpdat{f_u}{u_j}{u=0}$ for each $j=1,\ldots,d$.
\end{lem}

Denote by $c(f)$ the number of constant vector fields $\dpar{}{X_j}$ which do not belong to $\tke f$.
\begin{teo}\label{thm_dpsu}
Let $f\colon(\K^n,0)\to(\K^p,0)$ be a smooth map-germ with finite
$\Keq_e$-codimension. Then, $d:=\dim_\K N(f) = \Kecod(f) - c(f)$. Moreover,
if the classes of $\gamma_1,\ldots,\gamma_d\in\theta(f)$ form a
$\K$-basis of the quotient $N(f)$, then:
$$F(x,u) = \left(f(x) + \sum_{i=1}^d
u_i\gamma_i,u_1,\ldots,u_d\right)$$
is a minimal stable unfolding of $f$.
In particular, $f$ admits an OPSU if and only if $\dim_\K N(f) = 1$.
\end{teo}

\begin{proof}
The following argument is similar to the one found in Proposition 7.1 from \cite{nunomond}, but taking into account that $f$ is not of rank 0.

Notice that the only elements in $\omega f(\theta_p)$ that do not belong to $f^*\m_p\theta(f)$ are those from the subspace $\operatorname{Sp}_\K\left\lbrace\dpar{}{X_1},\ldots,\dpar{}{X_p}\right\rbrace$. Therefore $N(f)$ is effectively $\nke f$ but adding all the constants into the zero class.
Hence, the first equality is clear.

This also means that $\nke f$ is contained in the linear space
generated over $\K$ by the classes of $\dpar{}{X_1},\ldots,\dpar{}{X_p}$,
$\gamma_1,\ldots,\gamma_d$.
As a remark, since $f$ is of arbitrary rank, the class of some
$\dpar{}{X_i}$ might be zero over $\nke f$, and so these classes might not
form a system of generators.
Now, by Lemma \ref{lema_k_isom}, we have that $\nke F$ is contained
in the linear space generated over $\K$ by the classes of
$\dpar{}{X_1},\ldots,\dpar{}{X_p},\dpar{}{U_1},\ldots,\dpar{}{U_d}$.
Again, some of the $\dpar{}{X_i}$ might be in the zero class, but we
still have:
$$\tke F +
\operatorname{Sp}_\K\left\lbrace\dpar{}{X_1},\ldots,\dpar{}{X_p},\dpar{}{U_1},\ldots,\dpar{}{U_d}\right\rbrace = \theta(F)$$
Hence, $F$ is stable (see Theorem 4.1 from \cite{nunomond}).
No stable unfolding with less parameters can exist, since the image of the $\dpar{}{U_1},\ldots,\dpar{}{U_d}$ via the isomorphism must project to a basis of $N(f)$ for the unfolding to be stable.

\end{proof}

The $\Keq_e$-codimension is constant in a $\Keq$-orbit, so one could be inclined to think that having an OPSU depends only on the $\Keq$-orbit (i.e. on the multiplicity). However, $c(f)$ is not, as the following example shows (see also Section \ref{mult4}).

\begin{ex}
In Rieger's classification \cite{Rieger} amongst the multiplicity 4 germs, $(x,y^4+xy)$ and $(x,y^4+xy^2+y^{2k+1})$ ($k>1$) admit OPSUs, and $(x,y^4+x^2y+y^5)$ and $(x,y^4+x^2y)$ don't. It can be seen that $c(f)=2$ for the first two, but for the last two $(1,0)\in\tke f$ and $c(f)=1$.
\end{ex}

We finish the section by giving a normal form of the versal unfolding of map-germs that admit stable unfolding, involving $N(f)$.
Part of the argument in the following proof can be found as the first step in the proof of Theorem 5.1 from Ruas' thesis (\cite{ruasthesis}), which was later published in \cite{riegerruas}.
It is used there to obtain a sufficient condition (equivalent in the article) for the $\A$-orbit of a germ with finite singularity type to be open on its $\Keq$-orbit.
The extra hypothesis of finite $\A_e$-codimension allows to prove another result:

\begin{prop}
If $f\colon(\K^n,0)\to(\K^p,0)$ has finite $\A_e$-codimension $k$,
and $N(f) = \operatorname{Sp}_\K\lbrace \gamma_1,\ldots, \gamma_d
\rbrace$ for some $\gamma_i\in\theta(f)$, then:
$$\nae f =\operatorname{Sp}_\K\left\lbrace\gamma_1,\ldots,\gamma_d,
\sum_{i=1}^d p_i^1(f(x))\gamma_i,\ldots,\sum_{i=1}^dp_i^{k-d}(f(x))
\gamma_i\right\rbrace$$
for some $p_i^j\in\m_p$.
In particular,  $f$ admits an OPSU if and only if it admits a versal
unfolding of the form:
$$\left( f(x) + \left(\lambda_1 + \sum^{k}_{i=2}\lambda_ip_i(f(x))
\right)\gamma_1(x),\lambda \right)$$
for some $p_2,\ldots,p_k\in\m_p$.
\end{prop}
\begin{proof}

First notice that:
$$N(f) = \frac{\theta(f)}{\tae f + f^*\m_p\theta(f)}$$
Hence, since $\gamma_1,\ldots,\gamma_d$ are a $\K$-basis of $N(f)$,
they must form a $\K$-independent set over $\nae f$.
Since $\Aecod(f) = k < \infty$, there must exist some
$\gamma_{d+1},\ldots,\gamma_{k}$ that complete a basis of $\nae f$.

We can also use the Malgrange Preparation Theorem (see the proofs mentioned above from \cite{ruasthesis, riegerruas}) to see that $\nae f$ is generated by
$\gamma_1,\ldots,\gamma_d$ as an $\O_p$-module via $f$.
Therefore, each $\gamma_{d+1},\ldots,\gamma_k$ must be in the class
of some element in
$f^*\m_p\operatorname{Sp}_\K\lbrace\gamma_1,\ldots,\gamma_d\rbrace$ over $\nae f$, giving a basis in the required
form.

In particular, if $d =1$ there must exist $p_2,\ldots,p_k\in\m_p$ so
that $\gamma_i(x) = p_i(f(x))\gamma_1(x)$ for each $i=2,\ldots,k$,
and so a versal unfolding in the required form can be constructed.
The converse is also true, since finding a versal unfolding of this
form implies that $\nae f = \operatorname{Sp}_\K\lbrace\gamma_1,p_2(f(
x))\gamma_1,\ldots,p_k(f(x))\gamma_1\rbrace$.
As $p_i(f(x))\gamma_1\in f^*\m_p\theta(f)$ for each $i=2,\ldots,k$,
we have that $\dim_\K N(f) = \dim_\K\left(\operatorname{Sp}_\K\lbrace\gamma_1(x)\rbrace\right) = 1$ and so  by Theorem
\ref{thm_dpsu} $f$ admits an OPSU of the form $(f(x) + \lambda\gamma_1(x),\lambda)$.
\end{proof}

In spite of not knowing explicitly who the $p_i$ are, this result can be of great interest when dealing with equivalence of unfoldings.

\begin{ex}
Suppose you are given a germ and its versal unfolding, for example $P_2:(x,y,z^5+xz,z^3+yz)$ in Houston and Kirk's list \cite{houstonkirk} and its versal unfolding $(x,y,z^5+xz+\lambda_1z^2+\lambda_2xz^2,z^3+yz,\lambda_1,\lambda_2)$. An advanced reader might realise that $P_2$ is a 1-parameter unfolding of $H_2$ in Mond's list \cite{mond}, which has $\A_e$-codimension 2, so the versal unfolding of $H_2$ should be an OPSU of $P_2$. However, with less background on existing classifications, knowing whether $P_2$ admits an OPSU or not is a priori not direct. Now, the versal unfolding can be written as $(P_2,\lambda_1,\lambda_2)+(\lambda_1+\lambda_2x)(0,0,z^2,0,0,0)$, and by the above result, this implies that $(P_2,\lambda)+\lambda(0,0,z^2,0,0)=(x,y,z^5+xz+\lambda z^2,z^3+yz,\lambda)$ is an OPSU.
\end{ex}

\section{Corank 1 multiplicity 4 germs from $\K^3$ to $\K^3$.}\label{mult4}

All simple germs of multiplicity 4 in Marar and Tari's list \cite{marartari} admit an OPSU. In fact the germs $4_1^k:(x,y,z^4+xz+y^kz^2)$ and $4_2^k:(x,y,z^4+(y^2+x^k)z+xz^2)$ are augmentations and any augmentation admits an OPSU (see \cite{BROS} for details on augmentations and their simplicity). It is natural to ask whether there are any other finitely determined germs in this $\Keq$-orbit which do not admit OPSU.

\begin{teo}
There are infinitely many non $\A$-equivalent finitely determined germs of multiplicity 4 in $\C^3$ to $\C^3$ which do not admit an OPSU.
\end{teo}
\begin{proof}
Let $H(d_1,d_2,d_3)$ the set of all homogeneous polynomial mappings $F =(f_1,f_2,f_3):\C^3\to\C^3$, such that
$deg f_i = d_i$. By Theorem 1.1 in \cite{FJR}, if $gcd(d_1,d_2,d_3)=1$ and $gcd(d_i,d_j)\leq 2$ for $1\leq i\leq j\leq 3$, then there exists a non-empty Zariski open subset $U\subset H(d_1,d_2,d_3)$ such that for every mapping
$F\in U$ the map germ $(F,0)$ is finitely $\A$-determined. Taking $(d_1,d_2,d_3)=(1,1,p)$ with $p>4$, we then have that there exists a finitely determined map-germ with homogeneous entries of those degrees in each component. By linear changes of coordinates in the source, this germ will be equivalent to $(x,y,\phi_p(x,y,z))$, where $\phi_p$ is a homogeneous polynomial of degree $p$ in the variables $x,y,z$.

Now, by Lemma 1.2 in \cite{BRS}, if we add polynomials in each component of degrees strictly less than $(1,1,p)$, the resulting germ will be finitely determined. So we have that $f_p(x,y,z)=(x,y,\phi_p(x,y,z))+(0,0,z^4)=(x,y,z^4+\phi_p(x,y,z))$ is finitely determined. Notice that it has multiplicity 4 for any $p$.

Next, taking $p_1\neq p_2$, we show that $f_{p_1}$ is not $\mathscr A$-equivalent to $f_{p_2}$. Notice that $f_p(x,y,z)$ can be written as $$(x,y,z^4+\tilde\phi_p(x,y,z)z^4+\phi^1_p(x,y)z+\phi^2_p(x,y)z^2+\phi^3_p(x,y)z^3)=$$
$$(x,y,(1+\tilde\phi_p(x,y,z))z^4+\phi^1_p(x,y)z+\phi^2_p(x,y)z^2+\phi^3_p(x,y)z^3),$$
for $\tilde\phi_p,\phi^1_p,\phi^2_p,\phi^3_p$ homogeneous polynomials of degrees $p-4,p-1,p-2$ and $p-3$, respectively. Notice that $(1+\tilde\phi_p(x,y,z))$ is a unity. Any change of coordinates in the source will mantain some monomial of degree $p$. On the other hand, the only way to eliminate monomials of degree $p$ with changes of coordinates in the target is with the change $Z\to Z-\psi(X,Y)Z$, where $(X,Y,Z)$ are the coordinates in the target and $\psi(X,Y)$ is a polynomial of degree $p-4$. However, any monomial in $\psi(X,Y)Z$ of degree $p$ will have $z^4$ and the monomials of degree $p$ in $f_p$ have $z,z^2$ or $z^3$, so there is no way of eliminating all the monomials of degree $p$ by $\A$-equivalence. Hence $f_{p_1}$ is not $\mathscr A$-equivalent to $f_{p_2}$ if $p_1\neq p_2$.
\end{proof}

\begin{ex}
In \cite{marartari} a method to classificate corank 1 simple germs in the equidimensional case due to du Plessis is given. Any such germ is $\A$-equivalent to a germ of the form $(x_1,\ldots,x_{n-1},x_n^{n+1}+\sum_{i=1}^{n-1}P_i(x_1,\ldots,x_{n-1})x_n^i)$. Two such germs are $\A$-equivalent if and only if the corresponding $(P_1,\ldots,P_{n-1})$ are $\mathscr G$-equivalent, where $\mathscr G$ is a subgroup of $\Keq$. The $\A_e$-codimension of the germ is equal to the $\mathscr G_e$-codimension of $(P_1,\ldots,P_{n-1})$, and following \cite{marartari}, for multiplicity 4 germs of type $(x,y,z^4+P(x,y)z+Q(x,y)z^2)$, this is given by the codimension in $\O_2^2$ of the space generated by $\{(3P,2Q),(P_x,Q_x),(P_y,Q_y),(-4PQ^2,9P^2)-Q^2(3P,2Q),(-2PQQ_x,3PP_x),(-2PQQ_y,3PP_y)\}$.

The corresponding calculation shows that the germ $(x,y,z^4+(x^2-y^2)z+y^2z^2)$ has $\A_e$-codimension 4 and is therefore finitely determined. However, by Theorem \ref{thm_dpsu} it does not admit an OPSU.
\end{ex}

\section{Declarations}

On behalf of all authors, the corresponding author states that there is no conflict of interest.

\end{document}